\documentclass{article}
\usepackage{amsmath}
\usepackage{latexsym}
\usepackage{amsthm}
\usepackage{amsfonts}
\usepackage{amssymb}
\usepackage{tikz}
\usepackage{tikz-cd}
\usepackage{gensymb}
\usepackage{cancel}
\usepackage{setspace}
\usepackage{mathtools}
\usepackage{framed}
\usepackage{tgbonum}
\usepackage{hyperref}
\usepackage{graphicx}
\usetikzlibrary{decorations.pathreplacing,angles,quotes}
\date{}
\author{Adam Afandi}
\title{Linear Hyperelliptic Hodge Integrals}

\newtheorem{Definition}{Definition}

\newtheorem{Lemma}{Lemma}
\newtheorem{Theorem}{Theorem}
\newtheorem{Corollary}{Corollary}
\newtheorem{Example}{Example}
\newtheorem{Remark}{Remark}

\newcommand{\mbar}{\overline{\mathcal{M}}}
\newcommand{\bztwo}{\mathcal{B}\mathbb{Z}_2}
\newcommand{\aux}{\mbar_{0, kt}(\mathbb{P}^1\times\bztwo, 1)}

\begin{document}
\maketitle


\begin{abstract}

We provide a closed form expression for linear Hodge integrals on the hyperelliptic locus. Specifically, we find a succinct combinatorial formula for all intersection numbers on the hyperelliptic locus with one $\lambda$-class, and powers of a $\psi$-class pulled back along the branch map. This is achieved by using Atiyah-Bott localization on a stack of stable maps into the orbifold $[\mathbb{P}^1/\mathbb{Z}_2]$.

\end{abstract}

\tableofcontents


\pagebreak

\section{Introduction}

\subsection{Context, History, Motivation}

The moduli space $\overline{\mathcal{M}}_{g, n}$ of stable genus $g$ curves with $n$ marked points has been an object of interest since the pioneering work of Deligne and Mumford (\cite{DeligneMumford69}, \cite{Mumford83}). In order to gain a deeper understanding of this moduli space, a promising strategy is to probe its intersection theory  \\

A natural way to construct cycles $[\alpha] \in A^\bullet(\overline{\mathcal{M}}_{g, n})$ is to consider loci of curves with certain geometric properties. The locus of curves that we are primarily interested in is the \emph{hyperelliptic locus} i.e. the locus of curves that admit a degree 2 map to $\mathbb{P}^1$. We denote the hyperelliptic locus by $\overline{\mathcal{H}}_{g, 2g + 2} \subseteq \overline{\mathcal{M}}_{g, 2g + 2}$, where the marked points are Weierstrass points. There are two natural maps, the \emph{branch map} and \emph{source map}, which we informally and briefly discuss here. Each moduli point in $\overline{\mathcal{H}}_{g, 2g + 2}$ consists of the data $[C_g, \varphi: C_g \rightarrow (T, p_1, \ldots, p_{2g + 2})]$, where $[C_g] \in \overline{\mathcal{M}}_{g, 2g + 2}$, and $[T, p_1, \ldots p_{2g + 2}] \in \overline{\mathcal{M}}_{0, 2g + 2}$. The marked points on the latter curve are the branch points of the degree 2 map $\varphi$. Given a point in $\overline{\mathcal{H}}_{g. 2g + 2}$, we can either remember the source curve or the target curve of $\varphi$. This is summed up in the following diagram:

\begin{center}
\begin{tikzcd}
\overline{\mathcal{H}}_{g, 2g + 2} \arrow[d, "br"] \arrow[r, hook, "\iota"] & \overline{\mathcal{M}}_{g, 2g + 2} \\
\overline{\mathcal{M}}_{0, 2g + 2}
\end{tikzcd}
\end{center}

\noindent There is a slight variant on the above situation that we also consider. The points in $\overline{\mathcal{H}}_{g, 2g + 2, 2}$ will correspond to branched coverings as before, but the last two marked points are a conjugate pair of points that are interchanged under the hyperelliptic involution. In this situation, $br$ now maps to $\overline{\mathcal{M}}_{0, 2g + 3}$, where the last marked point is the image of the conjugate pair.   \\ 

Let $\mathbb{E}_g$ be the Hodge bundle over $\overline{\mathcal{M}}_{g, 2g + 2}$, and let $\mathbb{L}_i$ be the $i^{th}$ universal cotangent line bundle over $\overline{\mathcal{M}}_{0, 2g + 2}$. We define $\lambda_i := c_i(\mathbb{E}_g)$, and $\psi_i := c_1(\mathbb{L}_i)$. Our goal is to investigate the following two types of intersection numbers: 

\begin{align*}
& \int_{\overline{\mathcal{H}}_{g, 2g + 2}}br^*\left( \psi_1^{2g - 1 - i} \right) \lambda_i \\
& \int_{\overline{\mathcal{H}}_{g, 2g + 2, 2}}br^*\left( \psi_{2g + 3}^{2g - i} \right) \lambda_i
\end{align*}

\noindent These intersection numbers are examples of \emph{hyperelliptic Hodge integrals}. In this paper, we find a closed form expression for these integrals:

\begin{framed}
\begin{Theorem}
Let $\displaystyle e_i(x_1, \ldots, x_n) := \sum_{1 \leq j_1 < \ldots < j_i \leq n}x_{j_1}x_{i_2}\ldots x_{j_j}$ be the $i^{th}$ elementary symmetric function on $x_1, \ldots x_n$. Then

\begin{align*}
& \int_{\overline{\mathcal{H}}_{g, 2g + 2}}br^*\left( \psi_1^{2g - 1 - i} \right) \lambda_i = \left( \frac{1}{2} \right)^{i + 1}e_i(1, 3, \ldots, 2g - 1) \\
& \int_{\overline{\mathcal{H}}_{g, 2g + 2, 2}}br^*\left( \psi_{2g + 3}^{2g - i} \right) \lambda_i = \left( \frac{1}{2} \right)^{i + 1}e_i(2, 4, \ldots, 2g)
\end{align*}

\end{Theorem}
\end{framed}

 There has been much progress made in the computations of hyperelliptic Hodge integrals. In \cite{GeneratingFunctions}, Cavalieri computed the generating functions of all integrals of the form

\begin{equation*}
\int_{\overline{\mathcal{H}}_{g, 2g + 2}}br^*\left(\psi_1^{i - 1}\right)\lambda_{g - i}\lambda_{g} \\
\end{equation*}

\noindent In \cite{Wise}, Wise showed that 

\begin{equation*}
\int_{\overline{\mathcal{H}}_{g, 2g + 2}}\frac{(1 - \lambda_1 + \ldots + (-1)^g\lambda_g)^2}{1 - br^*\left(\frac{\psi_1}{2}\right)} = \left( \frac{-1}{4} \right)^g
\end{equation*}

\noindent In \cite{JPT}, Johnson et al found an algorithm to compute linear Hodge integrals over spaces of cyclic covers, in terms of Hurwitz numbers. Furthermore, using the results from \cite{CCIT09}, one can extract hyperelliptic Hodge integrals as coefficients of a twisted $J$-function.  \\

Despite the tremendous progress made in the calculation of these intersection numbers, a closed form expression, as stated in the theorem above, does not seem to follow directly from the work of previous authors. As such, the main result of this paper provides a simple, yet highly nontrivial, expression for linear Hodge integrals over the hyperelliptic locus. 


\subsection{Outline of Paper}

The main strategy that this paper employs is \emph{Atiyah-Bott localization} on $\aux$. We recognize the space $\overline{\mathcal{H}}_{g, 2g + 2}$ as the space of genus 0 stable maps into the stacky point $\bztwo$, all of whose marked points have nontrivial isotropy,

\begin{equation*}
\overline{\mathcal{H}}_{g, 2g + 2} \cong \mbar_{0, (2g + 2)t}(\bztwo)
\end{equation*}

\noindent The space $\overline{\mathcal{H}}_{g, 2g + 2, 2}$ is isomorphic to the space of genus 0 stable maps into $\mathcal{B}\mathbb{Z}_2$, but the last marked point has trivial isotropy,

\begin{equation*}
\overline{\mathcal{H}}_{g, 2g + 2, 2} \cong \mbar_{0, (2g + 2)t, 1u}(\bztwo)
\end{equation*}

\noindent In order to compute linear hyperelliptic Hodge integrals, we compute auxiliary integrals on the space $\aux$, which vanish for dimension reasons. The integrals that we are interested in appear as `vertex terms' in the localization computation. We find recursions relating these integrals, and check that the purported formula for the integrals satisfy the recursions, thus proving the desired result.

 In Chapter 2, we introduce the spaces $\aux$, and $\mbar_{0, kt, \ell u}(\bztwo)$. In Chapter 3, we explain our localization set up, in Chapter 4 we discuss our auxiliary integrals, and in the final Chapter, we put all of the pieces together in order to prove the theorem.
 
 
 \subsection{Future Work} 
 
To generalize the result in this paper, we can allow arbitrarily many insertions of $\lambda$-classes. The author will investigate this generalization in his PhD thesis.


\subsection*{Acknowledgements} 

This paper would not have been possible without my advisors, M. Shoemaker and R. Cavalieri. Their patience and willingness to listen to my arguments, and their feedback and comments on early drafts of this paper, have helped tremendously. 


\section{The Spaces $\aux$ and $\mbar_{0, kt, \ell u}(\bztwo)$}

The intersection numbers that we are interested in come from an auxiliary computation on a larger moduli stack: $\mbar_{0, kt}(\mathbb{P}^1 \times \bztwo, 1)$. Intersection numbers on this space are examples of \emph{orbifold Gromov-Witten invariants}. Foundational material on orbifold Gromov-Witten (GW) theory can be found in \cite{CR01} and \cite{AGV08}. We will not need the entire machinery of orbifold GW theory. Instead, we highlight enough of the theory in order to begin the relevant computations.\\

Let $\mbar_{0, k}(\mathcal{X}, d)$ be the moduli stack of genus 0 degree $d$ stable maps with $k$ marked points, into the orbifold $\mathcal{X}$. The \emph{inertia stack} of $\mathcal{X}$ (see [\cite{MarkandYP}, Section 1]), denoted $\mathcal{I}\mathcal{X}$, is the fiber product 

\begin{center}
\begin{tikzcd}
\mathcal{I}\mathcal{X} \arrow[r] \arrow[d] & \mathcal{X} \arrow[d, "\Delta"] \\
\mathcal{X} \arrow[r, "\Delta"] & \mathcal{X} \times \mathcal{X}
\end{tikzcd}
\end{center}

\noindent where $\Delta$ is the diagonal map. The product is taken in the $2$-category of stacks. The points in $\mathcal{I}\mathcal{X}$ can be identified with all pairs $(x, g)$, where $x \in \mathcal{X}$ and $g \in \text{Aut}_{\mathcal{X}}(x)$. In the case that $\mathcal{X} = [V/G]$, where $V$ is a smooth projective variety, and $G$ is a finite abelian group, we have

\begin{equation*}
\mathcal{I}\mathcal{X} = \coprod_{g \in G}[\mathcal{X}^g/G]
\end{equation*}

\noindent In this paper, $\mathcal{X} = [\mathbb{P}^1/\mathbb{Z}_2] = \mathbb{P} \times \bztwo$, where $\mathbb{Z}_2$ acts trivially on $\mathbb{P}^1$. Therefore, by the description of the inertia stack above,

\begin{equation*}
\mathcal{I}\mathcal{X} = \left(\mathbb{P}^1 \times \bztwo \right) \amalg \left( \mathbb{P}^1 \times \bztwo \right) := \mathcal{I}\mathcal{X}_0 \amalg \mathcal{I}\mathcal{X}_1
\end{equation*}

The main difference between ordinary GW theory and orbifold GW theory is that the source curves are allowed to be \emph{orbicurves}, that is, the marked points and the nodes are allowed to have non-trivial orbifold structure. Furthermore, the evaluation maps no longer land in the target space, but instead land in the \emph{rigidified} inertia stack,

\begin{equation*}
\text{ev}_i : \overline{\mathcal{M}}_{0, k}(\mathcal{X}, d) \rightarrow \overline{\mathcal{I}}\mathcal{X}
\end{equation*}

\noindent The definition of the rigidified inertia stack is technical, and we refer the reader to [\cite{AGV08}, Section 3] for details. However, as explained in [\cite{AGV08}, Section 6], even though there is not a well defined evaluation map from the stack of stable maps to the inertia stack, because there is an isomorphism between the cohomology groups of $\overline{\mathcal{I}}\mathcal{X}$ and $\mathcal{I}\mathcal{X}$, there is a well defined map

\begin{equation*}
\text{ev}_i^*: H^\bullet(\mathcal{I}\mathcal{X}) \rightarrow H^\bullet(\mbar_{0, k}(\mathcal{X}, d))
\end{equation*}

If $\mathcal{X} = [\mathbb{P}^1/\mathbb{Z}_2]$, since $\mathcal{I}\mathcal{X}$ only consists of two components, the marked points on the source curve are either 'untwisted' or 'twisted', i.e. maps to $\mathcal{I}\mathcal{X}_0$ or maps to $\mathcal{I}\mathcal{X}_1$. In the former case, the marked point has trivial isotropy, and in the latter, it has non-trivial isotropy. With this, we have the following definitions:

\begin{Definition}
The substack $\mbar_{0, kt, \ell u}(\mathbb{P}\times\mathbb{Z}_2, d) \subset \mbar_{0, k + l}(\mathbb{P}^1\times\mathbb{Z}_2, d)$ is defined to be the space of degree $d$ maps of genus $0$ curves into $[\mathbb{P}^1/\mathbb{Z}_2]$, in which the first $k$ marked points are twisted, and the last $\ell$ marked points are untwisted. Similarly, the substack $\mbar_{0, kt, \ell u}(\bztwo) \subset \mbar_{0, k + \ell}(\bztwo)$ is the space of degree $0$ maps of genus $0$ curves into the stack point $\bztwo = [pt./\mathbb{Z}_2]$, in which the first $k$ marked points are twisted, and the last $\ell$ marked points are untwisted.
\end{Definition}

\begin{Remark}
In the case that $\ell = 0$, we suppress $\ell$ from the notation, and simply indicate the number of twisted points.
\end{Remark}

Let $[\mathcal{C} \rightarrow \mathbb{P}^1 \times \bztwo] \in \mbar_{0, kt}(\mathbb{P}^1 \times \bztwo, 1)$. If we compose the map $\mathcal{C} \rightarrow \mathbb{P}^1 \times \bztwo$ with projection onto the $\bztwo$ factor, we get the map $\mathcal{C} \rightarrow \bztwo$, which is equivalent to the data of a principal $\mathbb{Z}_2$-bundle over $\mathcal{C}$ branched over the $k$ marked points. By Riemann-Hurwitz, the total space of this bundle is a curve $C$ of genus $g = \frac{k - 2}{2}$. 

\begin{Definition}
The \emph{Hodge bundle $\mathbb{E}_g$} over $\mbar_{0, kt}(\mathbb{P}^1 \times \bztwo, 1)$ is the vector bundle whose fiber over the point $[\mathcal{C} \rightarrow \mathbb{P}^1 \times \bztwo]$ is $\Omega^1(C)$, where $C$ is the ramified cover of $\mathcal{C}$ described above. The Chern classes of this vector bundle are denoted $\lambda_i := c_i(\mathbb{E}_g)$. 
\end{Definition}

\noindent If $[\mathcal{C} \rightarrow \bztwo] \in \mbar_{0, kt}(\bztwo)$, then again, the map $\mathcal{C} \rightarrow \bztwo$ is equivalent to the data of a degree 2 branched covering of $\mathcal{C}$, whose branch locus is the $k$ marked points on the source curve. Similarly, if $[\mathcal{C} \rightarrow \bztwo] \in \mbar_{0, kt, 1u}(\bztwo)$, the map $\mathcal{C} \rightarrow \bztwo$ is equivalent to a degree $2$ branched covering of $\mathcal{C}$, in which the branch locus is the $k$ twisted points on the source curve, but the \emph{last} marked point, which is untwisted/has trivial isotropy, is \emph{not} a branch point of the covering. The preimage of the untwisted point is a pair of conjugate points that are interchanged under the hyperelliptic involution. With this geometric description, we see that

\begin{align*}
& \overline{\mathcal{H}}_{g, 2g + 2} \cong \mbar_{0, (2g + 2)t}(\bztwo) \\
& \overline{\mathcal{H}}_{g, 2g + 2, 2}\cong \mbar_{0, (2g + 2)t, 1u}(\bztwo)
\end{align*}

\noindent Lastly, we also need to know the dimensions of the spaces  we've discussed (see \cite{Johnson14}, Section 1.1.3):

\begin{align*}
& \text{dim}\left( \aux \right) = k \\
& \text{dim}\left(\mbar_{0, kt, \ell u}(\bztwo)\right) = k - 3 + \ell
\end{align*}


\section{Atiyah-Bott Localization}

The computational tool we use to prove our main result is Atiyah-Bott localization on $\mbar_{0, kt}(\mathbb{P}\times\bztwo, 1)$. Whenever a space has a torus action, a classical result of Atiyah and Bott says that computing intersection numbers on this space amounts to 'localizing' to the torus-fixed locus. More precisely:

\begin{Theorem}\label{Atiyah-Bott}
(\cite{AB88}, Section 3) Let $\mathcal{X}$ be a projective variety with a $\mathbb{C}^*$-action. Let $\Gamma_1, \ldots, \Gamma_n$ be the irreducible components of the fixed locus of the action. Then for any $\alpha \in H^*(\mathcal{X}, \mathbb{Q})$,

\begin{equation*}
\int_\mathcal{X}\alpha = \sum_{\Gamma_i}\int_{[\Gamma_i]}\frac{\alpha\vert_{\Gamma_i}}{e(N_{\Gamma_i})}
\end{equation*}

\noindent where $e(\underline{\hspace{0.5cm}})$ denotes the Euler class, and $N_{\Gamma_i}$ is the normal bundle to $\Gamma_i$. 

\end{Theorem}

\begin{flushright}
$\blacksquare$
\end{flushright}

\noindent The $\mathbb{C}^*$-action on the coarse space of $\mathbb{P}^1 \times \bztwo$ is given by

\begin{equation*}
\lambda \cdot [x_0 : x_1] = [x_0 : \lambda x_1]
\end{equation*}

\noindent and this action induces a $\mathbb{C}^*$-action on $\aux$ by post composition.  \\

Throughout, we let $t$ be the \emph{equivariant parameter} of the $\mathbb{C}^*$-equivariant cohomology ring of $\mbar_{0, kt}(\mathbb{P}^1 \times \bztwo, 1)$ i.e. $H_{\mathbb{C}^*}^\bullet(pt.) = \mathbb{C}[t]$. The variable $t$ has a precise meaning. The classifying space of $\mathbb{C}^*$ is $\mathbb{P}^\infty$. This space has a tautological line bundle $\mathcal{O}_{\mathbb{P}^\infty}(-1)$, and $t$ is the first Chern class of the dual to this line bundle. \\

By Theorem \ref{Atiyah-Bott}, in order to compute intersection numbers on $\aux$, we need two ingredients: the components of the fixed point locus, and the Euler class to the normal bundle of each component. \\

The components of the fixed point locus are described as follows. If $[f: \mathcal{C} \rightarrow \mathbb{P}^1\times\bztwo] \in \aux$ is a fixed point of the $\mathbb{C}^*$-action, the map $f$ must send the marked points on the source curve to the $\mathbb{C}^*$-fixed points on $\mathbb{P}^1$, namely, $ 0 := [1 : 0]$ and $ \infty := [0 : 1]$. Since $f$ is a map of degree 1, the only way a marked point can be mapped to either $0$ or $\infty$ is if $f^{-1}(0)$ or $f^{-1}(\infty)$ is a marked point, or, the marked points lie on a contracted component over $0$ or $\infty$. In the case that a component is contracted to $0$ or $\infty$, in order for $f$ to still be a degree 1 map, this contracted component is attached, via a node, to a $\mathbb{P}^1$ that maps with degree 1 to the target. Because of this description, we can index/enumerate the components of the $\mathbb{C}^*$-fixed locus using simple graphs, which are called \emph{localization graphs} (\cite{GraberPandharipande}, Section 4).  \\

\begin{Definition}
A localization graph $\Gamma$ for $\mbar_{0, kt}(\mathbb{P}^1 \times \bztwo, 1)$ is a decorated graph with the following properties:

\begin{enumerate}
\item{$\Gamma$ has two vertices, denoted $v_0$ and $v_\infty$. They correspond to contracted components over $0$ and $\infty$, respectively.}
\item{$\Gamma$ has one edge connecting $v_0$ and $v_\infty$. This edge corresponds to the component mapping with degree 1 to the target, and we refer to this edge as the \emph{central component} of $\Gamma$}
\item{The vertices $v_0$ and $v_\infty$ can be incident to \emph{half edges}.We denote the set of half edges incident to $v_0$ and $v_\infty$ as $e_0$ and $e_\infty$, respectively. We require that $|e_0| + |e_\infty| = k$}
\item{The half edges are \emph{labelled}, i.e. there is a bijective map $\nu: e_0\cup e_\infty \rightarrow \{1, 2, \ldots, k\}$.}
\end{enumerate}

\end{Definition}

\noindent Each localization graph $\Gamma$ represents a fixed locus in $\aux$. The spaces $\mbar_\Gamma$ are isomorphic to components of the fixed loci, up to a difference in gerbe structure due to gluing at the nodes. Consequently, integrals over a component of the fixed locus may be computed as an integral over the corresponding $\mbar_\Gamma$, after correcting by a factor, known as the \emph{gluing factor} \cite{CC09}.  The space $\mbar_\Gamma$ is determined by the \emph{vertices} of the localization graph,  

\begin{equation*}
\mbar_{\Gamma} := \mbar_{v_0} \times \mbar_{v_\infty}
\end{equation*}

\noindent The spaces $\mbar_{v_0}$ and $\mbar_{v_\infty}$ are described as follows. If $|e_0|$ is odd, then $\mbar_{v_0} := \mbar_{0, (|e_0| + 1)t}(\bztwo)$, and if $|e_0|$ is even, $\mbar_{v_0} := \mbar_{0, |e_0|t, 1u}(\bztwo)$. Similarly, if $|e_\infty|$ is odd, $\mbar_{v_\infty} := \mbar_{0, (|e_\infty| + 1)t}(\bztwo)$, and if $|e_\infty|$ is even, $\mbar_{v_\infty} := \mbar_{0, |e_\infty|t, 1u}(\bztwo)$.

\begin{Example}
It's best to understand localization graphs by seeing some examples. Consider the space $\mbar_{0, 6t}(\mathbb{P}^1 \times \bztwo, 1)$. Below are two examples of localization graphs for this space:

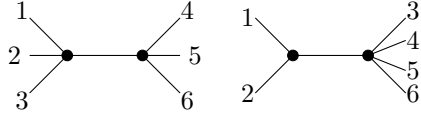
\begin{center}
\begin{tikzpicture}

\draw (0, 0) -- (1, 0);
\filldraw (0, 0) circle (2pt);
\filldraw (1, 0) circle (2pt);

\draw (0, 0) -- (-0.5, 0.5);
\draw (0, 0) -- (-0.5, 0);
\draw (0, 0) -- (-0.5, -0.5);

\draw (-0.6, 0.6) node {$1$};
\draw (-0.7, 0) node {$2$};
\draw (-0.6, -0.6) node {$3$};
\draw (1.6, 0.6) node {$4$};
\draw (1.7, 0) node {$5$};
\draw (1.6, -0.6) node {$6$};

\draw (1, 0) -- (1.5, 0.5);
\draw (1, 0) -- (1.5, 0);
\draw (1, 0) -- (1.5, -0.5);

\draw (3, 0) -- (4, 0);
\filldraw (3, 0) circle (2pt);
\filldraw (4, 0) circle (2pt);

\draw (3, 0) -- (2.5, 0.5);
\draw (3, 0) -- (2.5, -0.5);

\draw (4, 0) -- (4.5, 0.5);
\draw (4, 0) -- (4.5, 0.2);
\draw (4, 0) -- (4.5, -0.2);
\draw (4, 0) -- (4.5, -0.5);

\draw (2.4, 0.5) node {$1$};
\draw (2.4, -0.5) node {$2$};
\draw (4.6, 0.6) node {$3$};
\draw (4.6, 0.2) node {$4$};
\draw (4.6, -0.2) node {$5$};
\draw (4.6, -0.5) node {$6$};

\end{tikzpicture}
\end{center}

\noindent The graph on the left corresponds to a single component of the $\mathbb{C}^*$-fixed locus. Each map in this component has the following properties:

\begin{itemize}
\item{The map has a rational component mapping with degree 1 to the target, corresponding to the central component of $\Gamma$.}
\item{The rational component mapping with degree 1 is nodal to two rational components, one of which contracts to the point 0, and the other contracts to $\infty$.}
\item{The rational component that contracts to 0 contains the first 3 marked points, and the rational component that contracts to $\infty$ contains the last three marked points}
\end{itemize} 

\noindent Notice that the nodes are forced to have non-trivial isotropy. If we take a map in this fixed locus, and restrict the map to either of the contracted components, we would get an element in $\mbar_{0, 4t}(\bztwo)$; the first three marked points have non-trivial isotropy, but by Riemann-Hurwitz, the last marked point must also be twisted. Furthermore due to the subtleties that arise from working with stacks, when we integrate against these fixed loci, we must also take into account a \emph{gluing factor}. The precise formulation/derivation of gluing factors can be found in \cite{CC09}, but in the context of this paper, we can state the gluing factors explicitly in terms of the localization graph: each node contributes a factor of $2$ and the central component mapping with degree 1 contributes a factor of $\frac{1}{2}$. By this reasoning, if $\alpha \in A^\bullet\left(\mbar_{0, 6t}(\mathbb{P}^1 \times \bztwo, 1\right)$, we have

\begin{equation*}
\int_{\mbar_\Gamma}\frac{\alpha\vert_{\mbar_\Gamma}}{e(N_{\mbar_\Gamma})} = 2 \cdot \int_{\mbar_{0, 4t}(\bztwo) \times \mbar_{0, 4t}(\bztwo)}\frac{\alpha\vert_{\mbar_\Gamma}}{e(N_{\mbar_\Gamma})}
\end{equation*}

\noindent Similarly, for the localization graph on the right, we have

\begin{equation*}
\int_{\mbar_\Gamma}\frac{\alpha\vert_{\mbar_\Gamma}}{e(N_{\mbar_\Gamma})} = 2 \cdot \int_{\mbar_{0, 2t, 1u}(\bztwo) \times \mbar_{0, 4t, 1u}(\bztwo)}\frac{\alpha\vert_{\mbar_\Gamma}}{e(N_{\mbar_\Gamma)}}
\end{equation*}

\end{Example}

\begin{flushright}
$\square$
\end{flushright}

\noindent In general, localization graphs will correspond to products of smaller moduli spaces, which is precisely the reason why we obtain recursions of intersection numbers over these spaces. \\

\noindent We use the following Lemma to compute the (inverses of) Euler classes to $N_{\mbar_\Gamma}$

\begin{Lemma}\label{NormalBundle}
Let $\Gamma$ be a localization graph of $\aux$, and let us define the following subsets of vertices of $\Gamma$:

\begin{itemize}
\item{$\text{Val}_0(1) = \{\text{vertices of valence 1 over 0}\}$} 
\item{$\text{Val}_\infty(1) = \{\text{vertices of valence 1 over $\infty$}\}$}
\item{$\text{Val}_0(3) = \{\text{vertices of valence 3 over 0}\}$}
\item{$\text{Val}_\infty(3) = \{\text{vertices of valence 3 over $\infty$}\}$}
\item{$\text{Val}_0(\geq3) = \{\text{vertices of valence at least 3 over 0}\}$}
\item{$\text{Val}_\infty(\geq3) = \{\text{vertices of valence at least 3 over $\infty$}\}$}
\end{itemize}

Then we have

\pagebreak

\begin{framed}
\begin{equation*}
\frac{1}{e(N_{\mbar_\Gamma})} = \frac{\displaystyle (2)^{\left|\text{Val}_0(3)\right| + \left| \text{Val}_\infty(3) \right|}\left(\frac{1}{2}\right) \left(\prod_{v \in \text{Val}_0(1)} t\right)\left( \prod_{v \in \text{Val}_\infty(1)} -t \right)}{(-t^2)\left( \displaystyle \prod_{v \in \text{Val}_0(\geq 3)}t - \psi \right)\left( \displaystyle \prod_{v \in \text{Val}_\infty(\geq 3)}-t - \psi \right)}
\end{equation*}
\end{framed}

\end{Lemma}

\begin{flushright}
$\blacksquare$
\end{flushright}

\noindent See \cite{MS03} for a derivation of the above formula. We have incorporated the gluing factors into the formula; it corresponds to the term $(2)^{\left|\text{Val}_0(3)\right| + \left| \text{Val}_\infty(3) \right|}\left(\frac{1}{2}\right)$. \\

\noindent In the context of our computations, the Atiyah-Bott localization theorem can be stated as follows:

\begin{Corollary}\label{Atiyah-Bott-Moduli}
Let $\alpha \in A^\bullet(\aux)$. Then 

\begin{equation*}
\int_{\aux}\alpha = \sum_{\Gamma}\int_{\mbar_\Gamma}\frac{\alpha\vert_{\mbar_\Gamma}}{e(N_{\mbar_\Gamma})}
\end{equation*}

\noindent where the sum is over all localization graphs $\Gamma$ of $\aux$.
\end{Corollary}

\begin{flushright}
$\blacksquare$
\end{flushright}


\section{Auxiliary Integrals}

\noindent In this section, we set up two auxiliary integrals that are zero for dimension reasons. These integrals give us recursions for integrals over the spaces $\mbar_{0, kt}(\bztwo)$ and $\mbar_{0, kt, 1u}(\bztwo)$. 

\subsection{Setup} 

In order to ease notation, we make the following definition:

\begin{Definition}

\noindent For $k = 2g + 2$, we define

\begin{equation*}
D_{i, k} := \int_{\mbar_{0, kt}(\bztwo)}\psi_1^{k - 3 + i}\lambda_i = \int_{\overline{\mathcal{H}}_{g, 2g + 2}}br^*(\psi_1^{2g - 1 - i})\lambda_i
\end{equation*}

\begin{equation*}
d_{i, k} := \int_{\mbar_{0, kt, 1u}(\bztwo)}\psi_{k + 1}^{k - 2 + i}\lambda_i = \int_{\overline{\mathcal{H}}_{g, 2g + 2, 2}}br^*(\psi_{2g + 3}^{2g - i})\lambda_i
\end{equation*}

\end{Definition}

\noindent From previous results (\cite{JK01}, \cite{GeneratingFunctions}), we have the following initial values:

\begin{center}
\begin{framed}
$D_{1, 4} = \frac{1}{4}, \hspace{1cm} D_{0, k} = d_{0, k} = \frac{1}{2} \hspace{1cm} (k \ \text{even})$
\end{framed}
\end{center}

\noindent In order to begin the computations, we need the following two standard facts. If $\Gamma$ is a localization graph, then the classes $\text{ev}_i^*(0), \text{ev}_i^*(\infty)$ restrict to $\mbar_\Gamma$ as

\begin{equation*}
\text{ev}_i^*(0)\vert_{\mbar_\Gamma} = t, \ \text{ev}_i^*(\infty)\vert_{\mbar_\Gamma} = -t 
\end{equation*}

\noindent If $\mbar_\Gamma = \mbar_{0, k_1t, l_1u}(\bztwo) \times \mbar_{0, k_2t, l_2t}(\bztwo)$, we have

\begin{equation*}
\lambda_i\vert_{\mbar_\Gamma} = \sum_{i_1 + i_2 = i}\left( \lambda_{i_1}\vert_{\mbar_{0, k_1t, l_1u}(\bztwo)} \times \lambda_{i_2}\vert_{\mbar_{0, k_2t, l_2t}(\bztwo)} \right)
\end{equation*}


\subsection{First Auxiliary Integral}

Consider the following integral

\begin{center}
\begin{framed}
$I_A := \displaystyle \int_{\aux}\text{ev}_1^*(0)\text{ev}_2^*(0)\text{ev}_3^*(\infty)\lambda_i$
\end{framed}
\end{center}

\noindent Since $\text{dim}\left(\mbar_{0, kt}(\mathbb{P}^1 \times \bztwo, 1)\right) = k$, as long as $k \geq 6$, this integral vanishes. We already know the initial conditions 

\begin{equation*}
\int_{\mbar_{0, 4t}(\bztwo)}\lambda_1 = \frac{1}{4}, \int_{\mbar_{0, kt}(\bztwo)}\psi_i^{ k - 3} = \int_{\mbar_{0, kt, 1u}} \psi_{k + 1}^{k - 2} = \frac{1}{2}
\end{equation*}

\noindent and therefore, we do not need the auxiliary integral for $k = 4$. Since the product $\text{ev}_1^*(0)\text{ev}_2^*(0)\text{ev}_3^*(\infty)$ is in the integrand of $I_A$, when we use localization to compute $I_A$, all of the localization graphs that appear in the computation must have the first two marked points over $0$, and the third marked point over $\infty$. What remains is a systematic way to enumerate all of the localization graphs that have this property.

\begin{Definition}

Define $A^k_j$ to be the set of all localization of graphs of $\mbar_{0, kt}(\mathbb{P}^1 \times \bztwo, 1)$ with the following properties: The first two marked points lie over $0$, the third marked point is over $\infty$, there are $k - 3 - j$ marked points over $0$, and the remaining $j$ points are over $\infty$.

\end{Definition}

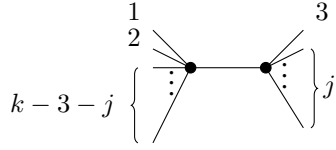
\begin{figure}[h]
\begin{center}
\begin{tikzpicture}

\filldraw (0, 0) circle (2pt);
\filldraw (1, 0) circle (2pt);

\draw (0, 0) -- (1, 0);

\draw (0, 0) -- (-0.5, 0.5);
\draw (0, 0) -- (-0.5, 0.25);
\draw (-0.75, 0.75) node {$1$};
\draw (-0.75, 0.4) node {$2$};
\draw (0, 0) -- (-0.5, 0); 
\draw (-0.25, -0.1) node {$\vdots$};
\draw (0, 0) -- (-0.5, -1);

\draw[decoration={brace,mirror},decorate] (-0.7, 0) -- (-0.7, -1);
\draw (-1.7, -0.5) node {$k - 3 - j$};

\draw (1, 0) -- (1.5, 0.5);
\draw (1.75, 0.75) node {$3$};

\draw (1, 0) -- (1.5, 0.25);

\draw (1.25, 0) node {$\vdots$};
\draw (1, 0) -- (1.5, -0.8);

\draw[decoration={brace},decorate] (1.6, 0.25) -- (1.6, -0.8);
\draw (1.85, -0.25) node {$j$};

\end{tikzpicture}
\end{center}
\caption{\emph{An element in $A_j^k$}}
\end{figure}

\noindent When we localize, by Corollary \ref{Atiyah-Bott-Moduli}, we have

\begin{equation*}
I_A = \sum_{j = 0}^{k - 3}\sum_{\Gamma \in A_j^k}(-t^3)\int_{\mbar_\Gamma}\frac{\lambda_i\vert_{\mbar_\Gamma}}{e(N_{\mbar_\Gamma})}
\end{equation*}

\noindent We evaluate this integral by computing the contributions coming from each set $A_j^k$

Let us first consider $A_0^k$ and $A_1^k$. The set $A_0^k$ only contains one localization graph $\Gamma$, and the corresponding moduli space is $\mbar_\Gamma = \mbar_{0, kt}(\bztwo)$.  So the contribution to $I_A$ coming from $A_0^k$ is

\begin{equation*}
(-t^3)\int_{\mbar_{0, kt}(\bztwo)}\frac{-\lambda_i}{t^2(t - \psi_k)} = \int_{\mbar_{0, kt}(\bztwo)}\frac{\psi_k^{k - 3 - i}}{t^{k - 3 - i}}\lambda_i = \frac{1}{t^{k - 3 - i}}D_{i, k}
\end{equation*}

\noindent Now let us consider $A_1^k$. There are ${k - 3 \choose 1} = k - 3$ localization graphs contained in $A_1^k$,  coming from the choice of labeling of the marked point over $\infty$. The moduli space corresponding to each of these localization graphs is $\mbar_\Gamma = \mbar_{0, (k - 2)t, 1u}(\bztwo) \times \mbar_{0, 2t, 1u}(\bztwo)$. The second factor is a moduli space of dimension 0, and when we integrate the fundamental class against it, we get $\frac{1}{2}$ ({\cite{JK01}, Proposition 3.4). So the contribution of $A_1^k$ to $I_A$ is

\begin{align*}
{k - 3 \choose 1}(-t^3) & \int_{\mbar_{0, (k - 2)t, 1u}}\frac{-\lambda_i}{t^2(t - \psi_{k - 1})} \int_{\mbar_{0, 2t, 1u}}\frac{1}{-t - \psi_3} \\
& = {k - 3 \choose 1}\left(\frac{-1}{t} \right)\int_{\mbar_{0, (k - 2)t, 1u}(\bztwo)}\frac{\psi_{k - 1}^{k - 4 - i}}{t^{k - 4 - i}}\lambda_i \\
&  = {k - 3 \choose 1}\frac{-1}{t^{k - 3 - i}}d_{i, k - 2}
\end{align*}

Now we need to consider $A_j^k$ for $j \geq 2$. Notice that the corresponding moduli spaces of each localization graph will vary, depending on whether $j$ is even or odd, so we consider both cases separately. In the case that $j$ is even, each localization graph of $A_j^k$ corresponds to the moduli space

\begin{equation*}
\mbar_\Gamma = \mbar_{0, (k - j)t}(\bztwo) \times \mbar_{0, (j + 2)t}(\bztwo)
\end{equation*}

\noindent There are $k - 3 \choose j$ localization graphs in $A_j^k$, so the contribution of $A_j^k$ to $I_A$ is

\begin{align*}
& {k - 3 \choose j}(-t^3)(2) \sum_{i_i + i_2 = i}\int_{\mbar_{0, (k - j)t}(\bztwo)}\frac{-\lambda_{i_1}}{t^2(t - \psi_{k - j})}\int_{\mbar_{0, (j + 2)t}}\frac{\lambda_{i_2}}{(-t - \psi_{j + 2})} \\
& = {k - 3 \choose j}\left(\frac{-1}{t}\right)(2)\sum_{i_1 + i_2 = i}\int_{\mbar_{0, (k - j)t}(\bztwo)}\frac{\psi_{k - j}^{k - 3 - j - i_1}}{t^{k - 3 - j - i_1}}\lambda_{i_1} \\
& \times \int_{\mbar_{0, (j + 2)t}(\bztwo)}(-1)^{j - 1 - i_2}\frac{\psi_{j + 2}^{j - 1 - i_2}}{t^{j - 1 - i_2}}\lambda_{i_2} \\
& = {k - 3 \choose j}\frac{-1}{t^{k - 3 - i}}(2)\sum_{i_1 + i_2 = i}(-1)^{j - 1 - i_2}D_{i_1, k - j}D_{i_2, j + 2} \\
& = {k - 3 \choose j}\frac{2}{t^{k - 3 - i}}\sum_{\ell = 0}^i(-1)^{\ell}D_{i - \ell, k - j}D_{\ell, j + 2}
\end{align*}

\noindent In the case that $j$ is odd, the localization graphs contained in $A_j^k$ correspond to

\begin{equation*}
\mbar_\Gamma = \mbar_{0, (k - 1 - j)t, 1u}(\bztwo) \times \mbar_{0, (j + 1)t, 1u}(\bztwo)
\end{equation*}

\noindent and the contribution of $A_j^k$ is 

\begin{align*}
& {k - 3 \choose j}(-t^3)(2) \sum_{i_i + i_2 = i}\int_{\mbar_{0, (k - 1 - j)t, 1u}(\bztwo)}\frac{-\lambda_{i_1}}{t^2(t - \psi_{k - j})}\int_{\mbar_{0, (j + 1)t, 1u}}\frac{\lambda_{i_2}}{(-t - \psi_{j + 2})} \\
& = {k - 3 \choose j}\left(\frac{-1}{t}\right)(2)\sum_{i_1 + i_2 = i}\int_{\mbar_{0, (k - 1 - j)t, 1u}(\bztwo)}\frac{\psi_{k - j}^{k - 3 - j - i_1}}{t^{k - 3 - j - i_1}}\lambda_{i_1} \\
& \times \int_{\mbar_{0, (j + 1)t, 1u}(\bztwo)}(-1)^{j - 1 - i_2}\frac{\psi_{j + 2}^{j - 1 - i_2}}{t^{j - 1 - i_2}}\lambda_{i_2} \\
& = {k - 3 \choose j}\frac{-1}{t^{k - 3 - i}}(2)\sum_{i_1 + i_2 = i} (-1)^{j - 1 - i_2}d_{i_1, k - 1 - j}d_{i_2, j + 1} \\
& = {k - 3 \choose j}\frac{-2}{t^{k - 3 - i}}\sum_{\ell = 0}^i(-1)^\ell d_{i - \ell, k - 1 - j}d_{\ell, j + 1}
\end{align*}

\noindent Recalling that $I_A = 0$, for $k \geq 6$ and $i \geq 1$, we obtain the following recursion:

\begin{framed}\label{FirstRecursion}
\begin{align*}
D_{i, k} = 2\sum_{j \ \text{odd}} & {k - 3 \choose j}\left( \sum_{\ell = 0}^i (-1)^\ell d_{i - \ell, k - 1 - j}d_{\ell, j + 1} \right) \\
& - 2\sum_{j \ \text{even}}{k - 3 \choose j}\left( \sum_{\ell = 0}^i (-1)^\ell D_{i - \ell, k - j}D_{\ell, j + 2} \right)
\end{align*}
\end{framed}


\subsection{Second Auxiliary Integral}

Now consider the integral

\begin{center}
\begin{framed}
$I_B:= \displaystyle \int_{\aux}\text{ev}_1^*(0)\text{ev}_2^*(0)\lambda_i$ 
\end{framed}
\end{center}

\noindent Again, $I_B = 0$ for dimension reasons, so evaluating it will result in a recursion as before. As in the previous auxiliary integral, we enumerate the relevant localization graphs.

\begin{Definition}
$B_j^k$ is the set of all localization graphs of $\aux$ with the following properties: the first two marked points lie over $0$, $j$ marked points lie over $\infty$, and the remaining $k - 2 - j$ marked points lie over $0$.
\end{Definition}

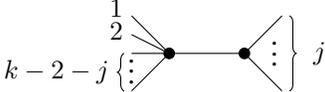
\begin{figure}[h]
\begin{center}
\begin{tikzpicture}

\filldraw (0, 0) circle (2pt);
\filldraw (1, 0) circle (2pt);

\draw (0, 0) -- (1, 0);

\draw (0, 0) -- (-0.5, 0.5);
\draw (0, 0) -- (-0.5, 0.25);
\draw (-0.7, 0.6) node {$1$};
\draw (-0.7, 0.3) node {$2$};

\draw (1, 0) -- (1.5, 0.5);
\draw (1, 0) -- (1.5, -0.5);
\draw (1.4, 0.10) node {$\vdots$};

\draw[decoration={brace},decorate] (1.6, 0.5) -- (1.6, -0.5);
\draw (2, 0) node {$j$};

\draw (0, 0) -- (-0.5, 0);
\draw (0, 0) -- (-0.5, -0.5);
\draw (-0.5, -0.13) node {$\vdots$};

\draw[decoration={brace,mirror},decorate] (-0.6, 0) -- (-0.6, -0.5);
\draw (-1.5, -0.25) node {$k - 2 - j$};

\end{tikzpicture}
\end{center}
\caption{\emph{An element in $B_j^k$}}
\end{figure}

\noindent By Corollary \ref{Atiyah-Bott-Moduli}, we see that

\begin{equation*}
I_B = \sum_{j = 0}^{k - 2}\sum_{\Gamma \in B_j^k}(t^2)\int_{\mbar_\Gamma}\frac{\lambda_i\vert_{\mbar_\Gamma}}{e(N_{\mbar_\Gamma})}
\end{equation*}

\noindent We begin by considering $B_0^k, B_1^k$, and $B_2^k$. There is one localization graph in $B_0^k$, whose corresponding moduli space is  $\mbar_\Gamma = \mbar_{kt, 1u}(\bztwo)$, so the contribution of $B_0^k$ to $I_B$ is

\begin{equation*}
(t^2)\int_{\mbar_{0, kt, 1u}(\bztwo)}\frac{-t}{-t^2(t - \psi_{k + 1})}\lambda_i = \int_{\mbar_{0, kt, 1u}(\bztwo)}\frac{\psi_{k + 1}^{k - 2 - i}}{t^{k - 2 - i}}\lambda_i = \frac{1}{t^{k - 2 - i}}d_{i, k}
\end{equation*}

\noindent Each localization graph of $B_1^k$ has corresponding moduli space $\mbar_\Gamma = \mbar_{0, kt}(\bztwo)$, and we get

\begin{align*}
{k - 2 \choose 1}(t^2) & \int_{\mbar_{0, kt}(\bztwo)}\frac{-1}{t^2(t - \psi_{k})} \\
&  = {k - 2 \choose 1}\left(\frac{-1}{t}\right)\int_{\mbar_{0, kt}(\bztwo)}\frac{\psi_k^{k - 3 - i}}{t^{k - 3 - i}}\lambda_i \\
& = {k - 2 \choose 1}\left(\frac{-1}{t^{k - 2 - i}} \right)D_{i, k}
\end{align*}

\noindent For $B_2^k$, each localization graph has corresponding moduli space $\mbar_\Gamma = \mbar_{0, (k - 2)t, 1u}(\bztwo) \times \mbar_{0, 2t, 1u}(\bztwo)$, and we get

\begin{align*}
{k - 2 \choose 2}(t^2)& \int_{\mbar_{0, (k - 2)t, 1u}(\bztwo)}\frac{-\lambda_i}{t^2(t - \psi_{k - 1})}\int_{\mbar_{0, 2t, 1u}(\bztwo)}\frac{1}{-t - \psi_3} \\
& = {k - 2 \choose 2}\frac{1}{t^2}\int_{\mbar_{0, (k - 2)t, 1u}(\bztwo)}\frac{\psi_{k - 1}^{k - 4 - i}}{t^{k - 4 - i}} \\
& = {k - 2 \choose 2}\frac{1}{t^{k - 2 - i}}d_{i, k - 2}
\end{align*}

Now lets consider $B_j^k$ for $j \geq 3$. As in the previous auxiliary integral, we consider the two subcases where $j$ is either even or odd. If $j$ is odd, each localization has corresponding moduli space $\mbar_\Gamma = \mbar_{0, (k - j + 1)t}(\bztwo) \times \mbar_{0, (j + 1)t}(\bztwo)$, and we get

\begin{align*}
{k - 2 \choose j}(t^2)(2) & \sum_{i_1 + i_2 = i}\int_{\mbar_{0, (k - j + 1)t}(\bztwo)}\frac{-\lambda_{i_1}}{t^2(t - \psi_{k - j + 1})}\int_{\mbar_{0, (j + 1)t}(\bztwo)}\frac{\lambda_{i_2}}{-t - \psi_{j + 1}} \\
& = {k - 2 \choose j}\frac{1}{t^2}(2)\sum_{i_1 + i_2 = i}\int_{\mbar_{0, (k - j + 1)t}(\bztwo)}\frac{\psi_{k - j + 1}^{k - j - 2 - i_1}}{t^{k - j - 2 - i_1}}\lambda_{i_1} \\
& \times \int_{\mbar_{0, (j + 1)t}(\bztwo)}(-1)^{j - 2 - i_2}\frac{\psi_{j + 1}^{j - 2 - i_2}}{t^{j - 2 - i_2}}\lambda_{i_2} \\
& =  {k - 2 \choose j}\frac{-1}{t^{k - 2 - i}}(2)\sum_{i_1 + i_2 = i}(-1)^{i_2}D_{i_1, k - j + 1}D_{i_2, j + 1} \\
& = {k - 2 \choose j}\frac{-2}{t^{k - 2 - i}}\sum_{\ell = 0}^i(-1)^\ell D_{i - \ell, k - j + 1}D_{\ell, j + 1}
\end{align*}

\noindent If $j$ is even, each localization graph has corresponding moduli space $\mbar_\Gamma = \mbar_{0, (k - j)t, 1u}(\bztwo) \times \mbar_{0, jt, 1u}(\bztwo)$, and we get

\begin{align*}
{k - 2 \choose j}(t^2)(2)& \sum_{i_1 + i_2 = i}\int_{\mbar_{0, (k - j)t, 1u}(\bztwo)}\frac{-\lambda_{i_1}}{t^2(t - \psi_{k - j + 1})}\int_{\mbar_{0, jt, 1u}(\bztwo)}\frac{\lambda_{i_2}}{-t - \psi_{j + 1}} \\
& = {k - 2 \choose j}\frac{1}{t^2}(2)\sum_{i_1 + i_2 = i}\int_{\mbar_{0, (k - j)t, 1u}(\bztwo)}\frac{\psi_{k - j + 1}^{k - j - 2 - i_1}}{t^{k - j - 2 - i_1}}\lambda_{i_1} \\
& \times \int_{\mbar_{0, jt, 1u}(\bztwo)}(-1)^{j - 2 - i_2}\frac{\psi_{j + 1}^{j - 2 - i_2}}{t^{j - 2 - i_2}}\lambda_{i_2} \\
& = {k - 2 \choose j}\frac{1}{t^{k - 2 - i}}(2)\sum_{i_1 + i_2 = i}(-1)^{i_2}d_{i_1, k - j}d_{i_2, j} \\
& = {k - 2 \choose j}\frac{2}{t^{k - 2 - i}}\sum_{\ell = 0}^i(-1)^\ell d_{i - \ell, k - j}d_{\ell, j}
\end{align*}

\noindent Since $I_B = 0$, for $k \geq 4$ and $i \geq 1$, all the previous computations combine to give the following recursion:

\begin{framed}\label{SecondRecursion}
\begin{align*}
d_{i, k} = 2\sum_{j \ \text{odd}} & {k - 2 \choose j}\left(\sum_{\ell = 0}^i(-1)^\ell D_{i - \ell, k - j + 1}D_{\ell, j + 1} \right) \\
& - 2\sum_{j \ \text{even}}{k - 2 \choose j}\left( \sum_{\ell = 0}^i(-1)^\ell d_{i - \ell, k - j}d_{\ell, j}\right)
\end{align*}
\end{framed}


\section{Proof of Main Theorem} 

With the two recursions obtained in the previous chapter, we are ready to prove our main theorem. The \emph{elementary symmetric functions} play a big role in our discussion.

\begin{Definition}
Let $x_1, \ldots, x_n$ be indeterminates. The $i^{th}$ elementary symmetric function on $x_1, \ldots, x_n$, denoted $e_i(x_1, \ldots, x_n)$, is defined as

\begin{equation*}
e_i(x_1, \ldots, x_n) = \sum_{1 \leq j_1 < j_2 < \ldots < j_i \leq n}x_{j_1}x_{j_2}\ldots x_{j_i}
\end{equation*}
\end{Definition}

\noindent First, we need some preliminary results.

\begin{Lemma}\label{CombinatorialVanishing}
For $p \leq n$,  we have the following identity:

\begin{equation*}
\sum_{k = 0}^{2n - 1}(-1)^k {2n - 1 \choose k}k^p = \sum_{k = 0}^{2n}(-1)^k{2n \choose k}k^p =  0
\end{equation*}

\end{Lemma}

\begin{proof}
The proof of this identity can be found in \cite{FP03} on pg. 30.
\end{proof}

\noindent Furthermore, we have a simple corollary to the above lemma

\begin{Corollary}\label{GeneralCombinatorialVanishing}
If $(m_1, \ldots, m_n) \in \mathbb{R}^n$, we have

\begin{equation*}
\sum_{k = 0}^{2n - 1}(-1)^k{2n - 1 \choose k}\prod_{i = 1}^n(m_i - k) = \sum_{k = 0}^{2n}(-1)^k{2n \choose k}\prod_{i = 1}^n(m_i - k) = 0
\end{equation*}

\end{Corollary}

\begin{proof}
We have

\begin{align*}
\sum_{k = 0}^{2n - 1}(-1)^k{2n - 1 \choose k}\prod_{i = 1}^n(m_i - k) & = \sum_{k = 1}^{2n - 1}(-1)^k{2n - 1 \choose k}\left( \sum_{i = 0}^ne_i(m_1, \ldots, m_n)k^{n - i} \right) \\
& = \sum_{i = 0}^ne_i(m_1, \ldots, m_n)\left( \sum_{k = 0}^{2n - 1}(-1)^k{2n - 1 \choose k}k^{n - i} \right) \\
& = 0
\end{align*}

\noindent A similar argument is used to show the vanishing of the second expression.

\end{proof}

\noindent We are now ready to state and prove our main theorem.

\begin{Theorem}

The following equalities hold for $D_{i, k}$ and $d_{i, k}$:

\begin{align*}
& D_{i, k} = \left(\frac{1}{2}\right)^{i + 1}e_i(1, 3, \ldots, k - 3) \\
& d_{i, k} = \left( \frac{1}{2} \right)^{i + 1}e_i(2, 4, \ldots, k - 2)
\end{align*}

\end{Theorem}

\begin{proof}
In order to prove the theorem, we simply need to check that the purported expressions of $D_{i, k}$ and $d_{i, k}$ satisfy the recursions obtained in the previous chapter. When we plug in the expressions into the recursion obtained in \ref{FirstRecursion}, we get

\begin{align*}
& \left(\frac{1}{2}\right)^{i + 1}e_i(1, 3, \ldots, k - 3) \\
& = 2 \sum_{j \ \text{odd}}{k - 3 \choose j}\left(\sum_{\ell = 0}^i(-1)^\ell \left(\frac{1}{2}\right)^{i + 2}e_{i - \ell}(2, 4, \ldots, k - 3 - j)e_{\ell}(2, 4, \ldots, j - 1) \right) \\
& - 2 \sum_{j \ \text{even}} {k - 3 \choose j} \left(\sum_{\ell = 0}^i (-1)^\ell \left(\frac{1}{2}\right)^{i + 2}e_{i - \ell}(1, 3, \ldots, k - 3 - j)e_{\ell}(1, 3, \ldots, j - 1) \right)
\end{align*}

\noindent Before we can proceed, we need a few standard combinatorial facts (see \cite{Stanley}, \cite{Macdonald}). The elementary symmetric functions have a very nice description via their generating functions,

\begin{equation*}
e_i(x_1, x_2, \ldots, x_n) = [t^i] \cdot \prod_{j = 1}^n(1 + x_jt)
\end{equation*} 

\noindent where by $[t^i] \cdot p(t)$, we mean the degree $i$ coefficient of $p(t)$. Furthermore, recall that, if $\displaystyle f(t) = \sum_{i \geq 0} a_it^i$ and $\displaystyle g(t) = \sum_{i \geq 0}b_it^i$ are generating functions for sequences $a_i$ and $b_i$, then $f(t)g(-t)$ is the generating function of the sequence $c_i$, where

\begin{equation*}
c_i = \sum_{j = 0}^i(-1)^ja_{i - j}b_j
\end{equation*}

\noindent Using the above facts, we see that

\begin{align*} 
& e_i(1, 3, \ldots, k - 3) = [t^i] \cdot \prod_{n = 1}^\frac{k - 2}{2}(1 + (2n - 1)t) \\
& \sum_{\ell = 0}^i(-1)^\ell e_{i - \ell}(2, 4, \ldots, k - 3 - j)e_\ell(2, 4, \ldots, j - 1) = [t^i] \cdot \prod_{n = 1}^{\frac{k - 3 - j}{2}}(1 + 2nt)\prod_{n = 1}^\frac{j - 1}{2}(1 - 2nt) \\
& \sum_{\ell = 0}^i(-1)^\ell e_{i - \ell}(1, 3, \ldots, k - 3 - j)e_\ell(1, 3, \ldots, j - 1)  = [t^i] \cdot \prod_{n = 1}^\frac{k - 2 - j}{2}(1 + (2n - 1)t)\prod_{n = 1}^\frac{j}{2}(1 - (2n - 1)t)
\end{align*}

\noindent and therefore, after a bit of simplification, when we plug in the expressions for $D_{i, k}$ and $d_{i, k}$ into the first recursion in \ref{FirstRecursion}, we see that the recursion is satisfied if and only if the following equality between polynomials holds:

\begin{align}\label{eqn}
& \prod_{n = 1}^\frac{k - 2}{2}(1 + (2n - 1)t) \\
& = \sum_{j \ \text{odd}}{k - 3 \choose j}\prod_{n = 1}^\frac{k - 3 - j}{2}(1 + 2nt)\prod_{n = 1}^\frac{j - 1}{2}(1 - 2nt) \nonumber \\
& - \sum_{j \ \text{even}}{k - 3 \choose j}\prod_{n = 1}^\frac{k - 2 - j}{2}(1 + (2n - 1)t)\prod_{n = 1}^\frac{j}{2}(1 - (2n - 1)t) \nonumber
\end{align}

\noindent where we define $\displaystyle \prod_{n =1}^0(1 - 2nt) := 1$. Having the foresight of eventually using the previous lemmas (as they are stated), we make the variable substitution $g := \frac{k - 2}{2}$, so that the above equation (\ref{eqn}) becomes

\begin{align*}
\prod_{n = 1}^g(1 + (2n - 1)t) & = \sum_{j \ \text{odd}}{2g - 1 \choose j}\prod_{n = 1}^\frac{2g - 1 - j}{2}(1 + 2nt)\prod_{n = 1}^\frac{j - 1}{2}(1 - 2nt) \\
&  - \sum_{j \ \text{even}}{2g - 1 \choose j}\prod_{n = 1}^\frac{2g - j}{2}(1 + (2n - 1)t)\prod_{n = 1}^\frac{j}{2}(1 - (2n - 1)t)
\end{align*}

\noindent Notice that

\begin{align*}
& j \ \text{odd} \implies \prod_{n = 1}^\frac{2g - 1 - j}{2}(1 + 2nt)\prod_{n = 1}^\frac{j - 1}{2}(1 - 2nt) = \prod_{n = 1}^{g}(1 + (2g - 1 - j - 2(n - 1))t) \\
& j \ \text{even} \implies \prod_{n = 1}^\frac{2g - j}{2}(1 + (2n - 1)t)\prod_{n = 1}^\frac{j}{2}(1 - (2n - 1)t) = \prod_{n = 1}^{g}(1 + (2g - 1 - j - 2(n - 1))t)
\end{align*}

\noindent We see that the desired result follows if the following polynomial vanishes

\begin{equation*}
P(t) := \sum_{j = 0}^{2g - 1}(-1)^j{2g - 1 \choose j}\prod_{n = 1}^{g}(1 + (2g - 1 - j - 2(n - 1))t)
\end{equation*}

\noindent Consider the following variable transformation for $P(t)$

\begin{equation*}
\widehat{P}(t) := t^gP\left(\frac{1}{t}\right) = \sum_{j = 0}^{2g - 1}(-1)^j{2g - 1 \choose j}\prod_{n = 1}^g((t + 2g - 1- 2(n - 1)) - j)
\end{equation*}

\noindent For $1 \leq n \leq g$, define $m_n(t) := t + 2g - 1 - 2(n - 1)$, so that $\widehat{P}(t)$ becomes

\begin{equation*}
\widehat{P}(t) = \sum_{j = 1}^{2g - 1} (-1)^j {2g - 1 \choose j}\prod_{n = 1}^g(m_n(t) - j)
\end{equation*}

\noindent By direct application of Corollary \ref{GeneralCombinatorialVanishing},

\begin{equation*}
\widehat{P}(1) = \widehat{P}(2) = \ldots = \widehat{P}(g + 1) = 0
\end{equation*}

\noindent Therefore, $\widehat{P}(t)$ has $g + 1$ distinct roots. But since the degree of $\widehat{P}(t)$ is $g$, it follows that $\widehat{P}(t) = 0$, and therefore, $P(t) = 0$. This shows that 

\begin{equation*}
D_{i, k} = \left( \frac{1}{2} \right)^{i + 1}e_i(1, 3, \ldots, k - 3)
\end{equation*}

\noindent In order to prove that $d_{i, k} = \left( \frac{1}{2} \right)^{i + 1}e_i(2, 4, \ldots, k - 2)$, we plug in this expression into the recursion in \ref{SecondRecursion}, and run through the same arguments as before. After going through the analogous computations as in the case for $D_{i, k}$, the desired result follows after showing the vanishing of the polynomial

\begin{equation*}
\sum_{j = 0}^{2g}(-1)^j{2g \choose j}\prod_{n = 1}^g(1 + (2g + 1 - j - 2(n - 1)))
\end{equation*}

\end{proof}

\end{document}